\documentclass[10pt]{amsart}

\usepackage{amsmath,amsthm,amsfonts,amscd,amssymb}
\usepackage{hyperref,wasysym}
\usepackage{verbatim}
\usepackage{amssymb}
\usepackage{graphicx}
\usepackage{color}
\usepackage{soul}

\newtheorem{thm}{Theorem}[section]
\newtheorem{cor}[thm]{Corollary}
\newtheorem{lem}[thm]{Lemma}

\newtheorem{ex}{Example}

\theoremstyle{definition}

\theoremstyle{remark}
\newtheorem{rem}{Remark}[section]

\begin{document}

\title{On indecomposable elements in lattices}

\author{Lenny Fukshansky}
\author{Filiana Kostopoulou}

\address{Department of Mathematics, 850 Columbia Avenue, Claremont McKenna College, Claremont, CA 91711}
\email{lenny@cmc.edu}
\address{Department of Mathematics, Pomona College, 610 N. College Ave, Claremont, CA 91711, USA}
\email{tkxc2022@mymail.pomona.edu}


\begin{abstract}
We study the distribution of indecomposable elements in Euclidean lattices. A positive element in a lattice is called indecomposable if it cannot be represented as a sum of two other positive nonzero elements. The set of all indecomposables in a lattice forms the Hilbert basis for the positive lattice semigroup. We classify lattices that contain only finitely many indecomposables versus those that contain infinitely many. In the two-dimensional case, we prove that every positive element in a lattice can be represented as a positive integer linear combination of at most two indecomposables, which is a certain variation of the discrete Carath\'eodory's property. In the case of lattices coming from ideals in the rings of integers of real quadratic number fields, we obtain an explicit counting estimate for the number of indecomposables with bounded norm, showing logarithmic growth.
\end{abstract}

\maketitle

\def\A{{\mathcal A}}
\def\B{{\mathcal B}}
\def\C{{\mathcal C}}
\def\D{{\mathcal D}}
\def\F{{\mathcal F}}
\def\x{{\mathcal H}}
\def\I{{\mathcal I}}
\def\J{{\mathcal J}}
\def\K{{\mathcal K}}
\def\L{{\mathcal L}}
\def\M{{\mathcal M}}
\def\N{{\mathcal N}}
\def\O{{\mathcal O}}
\def\R{{\mathcal R}}
\def\s{{\mathcal S}}
\def\V{{\mathcal V}}
\def\W{{\mathcal W}}
\def\X{{\mathcal X}}
\def\Y{{\mathcal Y}}
\def\H{{\mathcal H}}
\def\Z{{\mathcal Z}}
\def\OO{{\mathcal O}}
\def\BB{{\mathbb B}}
\def\cee{{\mathbb C}}
\def\EE{{\mathbb E}}
\def\Nn{{\mathbb N}}
\def\pee{{\mathbb P}}
\def\que{{\mathbb Q}}
\def\real{{\mathbb R}}
\def\zed{{\mathbb Z}}
\def\hyp{{\mathbb H}}
\def\aa{{\mathfrak a}}
\def\HH{{\mathfrak H}}
\def\qbar{{\overline{\mathbb Q}}}
\def\eps{{\varepsilon}}
\def\ahat{{\hat \alpha}}
\def\bhat{{\hat \beta}}
\def\gt{{\tilde \gamma}}
\def\h{{\tfrac12}}
\def\be{{\boldsymbol e}}
\def\bei{{\boldsymbol e_i}}
\def\bff{{\boldsymbol f}}
\def\ba{{\boldsymbol a}}
\def\bb{{\boldsymbol b}}
\def\bc{{\boldsymbol c}}
\def\bd{{\boldsymbol d}}
\def\bm{{\boldsymbol m}}
\def\bn{{\boldsymbol n}}
\def\bk{{\boldsymbol k}}
\def\bi{{\boldsymbol i}}
\def\bl{{\boldsymbol l}}
\def\bq{{\boldsymbol q}}
\def\bu{{\boldsymbol u}}
\def\bt{{\boldsymbol t}}
\def\bs{{\boldsymbol s}}
\def\bv{{\boldsymbol v}}
\def\bw{{\boldsymbol w}}
\def\bx{{\boldsymbol x}}
\def\bX{{\boldsymbol X}}
\def\bz{{\boldsymbol z}}
\def\bwy{{\boldsymbol y}}
\def\bY{{\boldsymbol Y}}
\def\bL{{\boldsymbol L}}
\def\baa{{\boldsymbol\alpha}}
\def\bbb{{\boldsymbol\beta}}
\def\bgg{{\boldsymbol\gamma}}
\def\bet{{\boldsymbol\eta}}
\def\bxi{{\boldsymbol\xi}}
\def\bo{{\boldsymbol 0}}
\def\bol{{\boldkey 1}_L}
\def\ep{\varepsilon}
\def\p{\boldsymbol\varphi}
\def\q{\boldsymbol\psi}
\def\rank{\operatorname{rank}}
\def\aut{\operatorname{Aut}}
\def\lcm{\operatorname{lcm}}
\def\sgn{\operatorname{sgn}}
\def\spn{\operatorname{span}}
\def\md{\operatorname{mod}}
\def\Norm{\operatorname{Norm}}
\def\dim{\operatorname{dim}}
\def\det{\operatorname{det}}
\def\Vol{\operatorname{Vol}}
\def\rk{\operatorname{rk}}
\def\Gal{\operatorname{Gal}}
\def\WR{\operatorname{WR}}
\def\WO{\operatorname{WO}}
\def\GL{\operatorname{GL}}
\def\pr{\operatorname{pr}}
\def\Tr{\operatorname{Tr}}
\def\dd{\partial}
\def\itt{\operatorname{int}}
\def\Ar{\operatorname{Area}}
\def\Aut{\operatorname{Aut}}

\section{Introduction and summary of results}
\label{intro}

Let $K$ be a totally real number field of degree $n \geq 2$ and let $\sigma_1,\dots,\sigma_n : K \hookrightarrow \real$ be its embeddings. The {\it norm} of an element $\alpha \in K$ is defined as
$$\Nn_K(\alpha) = \prod_{j=1}^n |\sigma_j(\alpha)|,$$
and the {\it house} of $\alpha$ is
$$H_K(\alpha) = \max \{ |\sigma_1(\alpha)|,\dots,|\sigma_n(\alpha)| \}.$$
Write $\O_K$ for the ring of integers of $K$ and define the {\it semigroup of totally positive integers} in $K$ to be
$$\O_K^+ = \{ \alpha \in \O_K : \sigma_j(\alpha) \geq 0\ \forall\ 1 \leq j \leq n \}.$$
A nonzero element $\alpha \in \O_K^+$ is called {\it decomposable} if there exist nonzero elements $\beta, \gamma \in \O_K^+$ such that $\alpha = \beta + \gamma$; otherwise, $\alpha$ is called {\it indecomposable}. Indecomposables in totally real number fields have been extensively studied, in particular in connection to the theory of universal quadratic forms (see~\cite{kala} for a detailed survey). As discussed in~\cite{kala}, the number of indecomposables in $\O_K^+$ is always finite, up to multiplication by totally positive units; in fact, indecomposables can be described as elements of $\O_K^+$ of appropriately bounded norm (\cite{kala_yatsyna}, Theorem~5). This being said, the group of units is infinite, and while multiplication by a unit does not change the norm of an element, it certainly changes its house. On the other hand, the ring of integers $\O_K$ can be viewed as a Euclidean lattice in $\real^n$ under the standard {\it Minkowski embedding}
$$\Sigma_K = (\sigma_1,\dots,\sigma_n) : K \hookrightarrow \real^n,$$
where $\O_K^+$ becomes a positive lattice semigroup, which allows us to view indecomposables in the context of the geometry of numbers. In particular, the house $H_K(\alpha)$ of an element $\alpha \in K$ becomes the sup-norm $|\Sigma_K(\alpha)|$ of its image under the Minkowski embedding. The goal of the present note is to define indecomposables more generally in Euclidean lattices, study their properties using geometric tools, and then apply our observations in the context of real quadratic number fields.

Throughout this paper, we write $|\ |$ for the sup-norm and $\|\ \|$ for the Euclidean norm on vectors in~$\real^n$. Let $L \subset \real^n$ be a lattice of rank $n \geq 2$ and define the positive lattice semigroup
$$L^+ = L \cap \real^n_{\geq 0} = \left\{ \bx \in L : x_i \geq 0\ \forall\ 1 \leq i \leq n \right\}.$$
A nonzero vector $\bx \in L^+$ is called {\it decomposable} if there exist nonzero vectors $\bwy,\bz \in L^+$ such that $\bx = \bwy + \bz$; otherwise, $\bx$ is called {\it indecomposable}. 

The first question we want to address is which lattices have finitely many and which infinitely many indecomposables. We start with some notation. A basis for $L$ is called a {\it positive basis} if it is contained in $L^+$. A lattice is called {\it rectangular} if it has an orthogonal basis. A lattice is called {\it virtually rectangular} if it has a rectangular sublattice of finite index; virtually rectangular lattices have been investigated in~\cite{kuhnlein} and~\cite{lps}. Let us also say that a lattice $L$ is {\it positive rectangular}, abbreviated PR, if it contains a positive orthogonal basis, and $L$ is called {\it positive virtually rectangular}, abbreviated PVR, if it contains a PR sublattice $M$ of finite index. We can now state our first result.

\begin{thm} \label{inf_many} If $L$ is a PVR lattice of rank $n$, then $L^+$ contains at most $d+n-1$ indecomposable elements, where
$$d = \min \{ [L:M] : M \subseteq L \text{ is PR} \}.$$
Otherwise, $L^+$ contains infinitely many indecomposable elements.
\end{thm}

\noindent
We prove Theorem~\ref{inf_many} in Section~\ref{ind_prop}, where we also review some properties of lattice semigroups and cones, and prove that vectors corresponding to positive successive minima in a lattice are indecomposable. In fact, the indecomposables are precisely the unique minimal generating set for the semigroup $L^+$, called the {\it Hilbert basis} for $L^+$ (Corollary~\ref{no_fin}). We briefly mention irreducible elements in lattice cones, observing that indecomposables are always irreducible but not vice versa. Each such cone contains only finitely many irreducibles, and they form the Hilbert basis for the lattice-point semigroup of this cone. Thus, it follows that while the semigroup $L^+$ of a non-PVR lattice $L$ contains infinitely many indecomposables, every positive lattice cone of $L$ contains only finitely many of these indecomposables.

In Section~\ref{ind_plane}, we focus on indecomposable elements in planar lattices, showing in particular that consecutive pairs of such indecomposables always form a basis for the lattice. The main result of this section is the following theorem.

\begin{thm} \label{ind_sum} Let $L \subset \real^2$ be a lattice of full rank and let $\bo \neq \bz \in L^+$ be a decomposable element. Then there exist indecomposables $\bu,\bv \in L^+$ and positive integers $\alpha, \beta$ so that $\bz = \alpha \bu + \beta \bv$.
\end{thm}

\noindent
In other words, every positive point of a lattice in $\real^2$ can be represented as a positive linear combination of at most two indecomposables. This result should be viewed in the context of Carath\'eodory's theorem which states that every point in a convex cone in $\real^n$ can be represented as a positive linear combination of at most $n$ of the cone generators. Furthermore, any point in a discrete cone (set of all integer points in a convex cone spanned by integer points) can be represented as a positive linear combination of irreducibles in this cone, which are precisely the elements of the Hilbert basis for the cone. This cone is said to have the {\it discrete Carath\'eodory's property} if such representation is possible by no more than $n$ elements of the Hilbert basis for every element. It is known that every discrete cone  in dimensions $n \leq 3$ has discrete Carath\'eodory's property, while it is no longer true in general for cones in dimensions $n \geq 6$; see~\cite{bruns} for details. Our Theorem~\ref{ind_sum} is the 2-dimensional analogue of the discrete Carath\'eodory's property for indecomposables in the positive orthant of a lattice: the key difference here, complicating things, is that $\real^2_{\geq 0}$ is not generated by vectors of $L$ as a cone unless $L$ is PVR.

We separately consider indecomposable elements in lattices coming from integral ideals in real quadratic number fields, in which case we can obtain a counting estimate. 
Let~$K$ be a real quadratic field with ring of integers~$\O_K$ and discriminant~$\Delta_K$. Let $J \subseteq \O_K$ be an ideal, and write $\Nn_K(J)$ for the norm of~$J$. Let $L = \Sigma_K(J)$, then $\det(L) = \Nn_K(J) \sqrt{|\Delta_K|}$. In this case, $L$ contains infinitely many indecomposables, hence we can define a counting function
\begin{eqnarray}
\label{cnt_funct}
\N_{K,J}(T) & = & \left\{ \bx \in L^+ : \bx \text{ indecomposable}, |\bx| \leq T \right\} \nonumber \\
& = & \left\{ \alpha \in J^+ : \alpha \text{ indecomposable}, H_K(\alpha) \leq T \right\},
\end{eqnarray}
for a positive real number $T$. In Section~\ref{counting}, we produce the following estimate.

\begin{thm} \label{main_cnt} Let the notation be as above, then
$$\N_{K,J}(T) \leq \left( \log \left( \frac{1+\sqrt{4 \det(L)^2+1}}{2 \det(L)} \right) \right)^{-1} (2 \log (\det(L)^2 T) + 1) + 1.$$
\end{thm}

\noindent
To prove this theorem, we first show that all indecomposables are contained under the hyperbola $y = \frac{\det(L)^2}{x}$ in the plane. This observation is based on an elegant argument of Kala and Yatsyna~\cite{kala_yatsyna}, which employs Minkowski's Convex Body Theorem. We then obtain the counting estimate by using the basis property of consecutive indecomposables along with an area computation. We are now ready to proceed. 

\bigskip

\section{Properties of indecomposable elements in lattices}
\label{ind_prop}

In this section we develop some necessary notation and prove Theorem~\ref{inf_many}. A vector $\bx \in L$ is called {\it primitive} if it is not a scalar multiple of another vector in $L$. Indecomposable vectors are primitive, however the converse is not necessarily true. It is our goal to understand the properties of indecomposable elements in $L^+$. We start with two simple but useful lemmas. Let us say that $\bx \leq \bwy$ (respectively, $\bx < \bwy$) if $x_i \leq y_i$ (respectively, $x_i < y_i$) for every $i$.

\begin{lem} \label{less} Let $L \subset \real^n$ be a lattice of rank $n$ and let $\bx \in L^+$. Then $\bx$ is indecomposable if and only if there does not exist $\bo \neq \bwy \in L^+$ such that $\bx \neq \bwy$ and $\bwy \leq \bx$.
\end{lem}

\proof
Suppose there exists $\bo \neq \bwy \in L^+$ such that $\bwy \leq \bx$. Let $\bz = \bx-\bwy$, then $\bz \in L^+$. Hence, $\bx = \bwy+\bz$. Now, suppose no such $\bwy$ exists. If $\bx$ is decomposable, then there exist some nonzero $\bwy,\bz \in L^+$ such that 
$$\bx = \begin{pmatrix} x_1 \\ \vdots \\ x_n \end{pmatrix} = \begin{pmatrix} y_1+z_1 \\ \vdots \\ y_n + z_n \end{pmatrix} = \bwy + \bz,$$
and so $x_i = y_i + z_i \geq y_i$ for each $1 \leq i \leq n$, i.e., $\bwy \leq \bx$, a contradiction. Hence, $\bx$ is indecomposable.
\endproof

\begin{lem} \label{sum} Every element in $L^+$ is a positive integer linear combination of finitely many indecomposables.
\end{lem}

\proof
Let $\bx \in L^+$ and suppose it is not indecomposable, then there exist $\bx_1, \bx_2 \in L^+$ such that $\bx = \bx_1 + \bx_2$, and so
$$\|\bx_1\|, \|\bx_2\| < \|\bx\|.$$
We can continue decomposing $\bx_1,\bx_2$ in the same manner, each time obtaining elements in $L^+$ of smaller norm. Since a lattice $L$ is a discrete subset of $\real^n$, there are only finitely many elements in $L$ of norm $\leq \|\bx\|$. Hence, this process must terminate, meaning that we obtain a positive integer linear combination of indecomposables equal to $\bx$.
\endproof

Let $X = \{ \bx_1,\dots,\bx_n \} \subset L^+$ be a collection of linearly independent vectors. We define the {\it cone of $X$} to be
$$C(X) := \left\{ \sum_{i=1}^n a_i \bx_i : a_1,\dots,a_n \in \real_{\geq 0} \right\}$$
and the {\it semigroup of $X$} to be
$$S(X) := \left\{ \sum_{i=1}^n a_i \bx_i : a_1,\dots,a_n \in \zed_{\geq 0} \right\}.$$
If $X$ is a basis for $L$, we refer to $C(X)$ and $S(X)$ as a {\it basis cone} and a {\it basis semigroup}, respectively. Properties of such positive cones and semigroups have been investigated in~\cite{lf_sw}. In particular, Lemma~2.1 of~\cite{lf_sw} guarantees that $L$ contains infinitely many positive bases; more precisely, this lemma can be formulated as follows.

\begin{lem} \label{prim_basis} Let $\bx_1 \in L^+$ be a primitive vector. Then there exists a positive basis $\bx_1,\bx_2,\dots,\bx_n \in L^+$.
\end{lem}

\noindent
The following observation will also be important to us; it is Lemma~2.2 of~\cite{lf_sw}.

\begin{lem} \label{cone_sem} Let $X$ be a positive basis for $L$. Then
$$L^+ \cap C(X) = S(X).$$
\end{lem}

\noindent
Next, let $X = \{ \bx_1,\dots,\bx_n \} \subset L$ be a collection of linearly independent points in a full-rank lattice $L$. Define the parallelepiped 
\begin{equation}
\label{PX}
P(X) = \left\{ \sum_{i=1}^n a_i \bx_i : 0 \leq a_i < 1\ \forall\ 1 \leq i \leq n \right\}.
\end{equation}
Let $M(X) = \spn_{\zed} X$ be a finite-index sublattice of $L$ and define
$$\det(X) := \det(M(X)) = \Vol_n(P(X)).$$

\begin{lem} \label{empty} The number of points of $L$ in $P(X)$ is equal to the index $[L:M(X)]$, i.e.,
$$|P(X) \cap L| = \frac{\det(X)}{\det(L)}.$$
In particular, the collection $X$ is a basis for $L$ if and only if $P(X) \cap L = \{\bo\}$.
\end{lem}

\proof
The index $[L:M(X)] = \frac{\det(X)}{\det(L)}$ is the number of cosets of $M(X)$ in $L$ and $P(X)$ is a fundamental domain for $M(X)$ in $\real^n$, meaning that $P(X) \cap L$ is the full set of coset representatives of $M(X)$ in $L$. Then $X$ is a basis for $L$ if and only if $L = M(X)$, i.e., if and only if $[L:M(X)] = 1$. In other words, $X$ is a basis if and only if the only coset representative of $M(X)$ in $L$ is~$\bo$.
\endproof

We are now ready for the proof of the theorem.

\proof[Proof of Theorem~\ref{inf_many}]
First suppose that $L \subset \real^n$ is a PR lattice and let $X = \{\bx_1,\dots,\bx_n\} \subset L^+$ be the corresponding positive orthogonal basis. Then these basis vectors must be multiples of the standard basis vectors $\be_1,\dots,\be_n$ in $\real^n$; we can assume without loss of generality that $\bx_i = t_i \be_i$ with $t_i \in \real_{>0}$ for $1 \leq i \leq n$. Let $\bwy \in L^+$. Since $X$ is a basis for $L$,
$$\bwy = \sum_{i=1}^n a_i t_i \be_i,$$
for some integers $a_1,\dots,a_n$ and for each $1 \leq i \leq n$, $y_i = a_i t_i \geq 0$. Then all $a_i \geq 0$, and so
$$L^+ = S(X).$$
Hence, $\bwy \in L^+$ is indecomposable if and only if $\bwy = \bx_i$ for some $1 \leq i \leq n$. This implies that $L^+$ contains precisely $n$ indecomposable elements.

Now, assume that $L$ is a PVR lattice and $M \subset L$ is a PR sublattice of $L$ of index~$d$. Let $X = \{ t_1 \be_1,\dots,t_n \be_n \}$ be the positive orthogonal basis for $M$. Let $P(X)$ be the parallelepiped spanned by $X$, as in~\eqref{PX}. Then every indecomposable element of $L$ must be contained in $P(X)$. Suppose not. Say, $\bwy$ is some indecomposable in $L$ which is outside of $P(X)$. Then $\bwy$ must have some $k$-th coordinate which is larger then $t_k$. But then the vector $\bz = \bwy - t_k \be_k$ is in $L^+$, which contradicts the assumption that $\bwy$ is indecomposable. Hence, the total number of indecomposables in $L$ is no larger than the number of nonzero points of $L$ in $P(X)$ + the number of points of $X$. Lemma~\eqref{empty} gives that the number of points of $L$ in P(X) is $d = [L:M]$, one of which is $\bo$. Thus, we have the upper bound: $d+n-1$.

Next, suppose that $L$ is not PVR. Arguing towards a contradiction, assume that $L$ has only finitely many indecomposables, let $V = \{ \bv_1,\dots,\bv_m \} \subset L^+$ be their complete set. Let $C(V)$ be the cone spanned by these indecomposables, i.e.,
$$C(V) := \left\{ \sum_{i=1}^m a_i \bv_i : a_1,\dots,a_m \in \real_{\geq 0} \right\} \subset \real^n_{\geq 0}$$
and let $C(V)^* = \real^n_{\geq 0} \setminus C(V)$. Since $L$ is not PVR, $C(V)^*$ must contain an unbounded cone of positive measure with apex at the origin. Let us write $\BB_n(r)$ for the ball of radius $r$ centered at the origin in $\real^n$. Then $C(V)^*$ contains Euclidean balls of arbitrarily large radius; in particular, it contains some Euclidean ball $B$ of radius at least $\mu(L)$, where
\begin{equation}
\label{cov_rad}
\mu(L) = \min \left\{ r \in \real_{> 0} : L + \BB_n(r) = \real^n \right\}
\end{equation}
is the covering radius of $L$. Then $B$ must contain a point of $L$, call it $\bwy$. Since $\bwy \not\in C(V)$, it cannot be represented as a nonnegative linear combination of elements of $V$. On the other hand, Lemma~\ref{sum} implies that there must exist a representation of $\bwy$ as a positive integer linear combination of indecomposables, meaning that there must exist indecomposables outside of the set $V$, a contradiction. This concludes the proof.
\endproof

\begin{ex} Let $L_1 = \begin{pmatrix} 3 & 1 \\ 1 & 1 \end{pmatrix} \zed^2$ and $L_2 = \begin{pmatrix} 6 & 2 \\ 1 & 1 \end{pmatrix} \zed^2$ be lattices of rank $n=2$. Both of these are PVR lattices, and their PR sublattices of minimal indices $d_1 = 2$ and $d_2=4$ are $M_1 = \begin{pmatrix} 2 & 0 \\ 0 & 2 \end{pmatrix} \zed^2$ and $M_2 = \begin{pmatrix} 4 & 0 \\ 0 & 4 \end{pmatrix} \zed^2$, respectively. Indecomposable elements in $L_1$ are $\begin{pmatrix} 2 \\ 0 \end{pmatrix}$, $\begin{pmatrix} 0 \\ 2 \end{pmatrix}$, and $\begin{pmatrix} 1 \\ 1 \end{pmatrix}$, so their number is $3 = d_1+n-1$. On the other hand, indecomposable elements in $L_2$ are $\begin{pmatrix} 4 \\ 0 \end{pmatrix}$, $\begin{pmatrix} 0 \\ 4 \end{pmatrix}$, and $\begin{pmatrix} 1 \\ 1 \end{pmatrix}$, so their number is $3 < d_2+n-1$. The difference stems from the fact that the corresponding parallelepiped $P(X)$ in the case of $L_2$ contains two other interior lattice points $\begin{pmatrix} 2 \\ 2 \end{pmatrix}$ and $\begin{pmatrix} 3 \\ 3 \end{pmatrix}$, which are not indecomposable.
\end{ex}
\smallskip

The {\it Hilbert basis} of an affine semigroup is the unique minimal generating set. We record here an immediate consequence of Lemma~\ref{sum} and Theorem~\ref{inf_many}, characterizing indecomposables as the Hilbert basis for the semigroup $L^+$.

\begin{cor} \label{no_fin} The set of indecomposables in the semigroup $L^+$ is precisely the Hilbert basis for $L^+$, which is finite if $L$ is a PVR lattice and infinite otherwise.
\end{cor}

\noindent
More generally, a nonzero element $\bwy \in C(X) \cap L$ is called {\it irreducible} if there do not exist nonzero elements $\bz_1,\bz_2 \in C(X) \cap L$ such that $\bwy = \bz_1+\bz_2$. The set of irreducible points forms the Hilbert basis for the semigroup $C(X) \cap L$. This definition implies that indecomposable elements of $L$ are irreducible in every positive cone spanned by a collection $X \subset L^+$, but not vice versa.
\smallskip

We can also obtain a bound on the Euclidean norm of the shortest indecomposable element in $L^+$. For this, let us define the {\it positive successive minima} of $L$ to be
\begin{equation}
\label{s_min}
\lambda^+_j(L) := \min \left\{ r \in \real_{>0} : \dim_{\real} \spn_{\real} \left( \BB_n(r) \cap L^+ \right) \geq j \right\},
\end{equation}
for each $1 \leq j \leq n$, so $0 < \lambda^+_1(L) \leq \dots \leq \lambda^+_n(L)$. We say that linearly independent vectors $\bx_1,\dots,\bx_n \in L^+$ correspond to these positive successive minima if $\|\bx_j\| = \lambda^+_j(L)$. Positive successive minima with respect to the unit cube instead of the unit ball were considered in~\cite{lf_sw}. Since $\|\bx\| \leq \sqrt{n} |\bx|$ for any vector $\bx \in \real^n$, Theorem~1.2 of~\cite{lf_sw} implies, in particular, that
\begin{equation}
\label{s_min_1}
\lambda^+_1(L) \leq \sqrt{n} \left( 2\mu(L) + 1 \right),
\end{equation}
where $\mu(L)$ is the covering radius of $L$ defined in~\eqref{cov_rad}.

\begin{lem} \label{s_min_ind} Let $\bx_1,\dots,\bx_n \in L^+$ be vectors corresponding to the positive successive minima $\lambda^+_1(L),\dots,\lambda^+_n(L)$, respectively. Then they are indecomposable.
\end{lem}

\proof
Suppose some $\bx_j$ is decomposable and let $\bwy,\bz \in L^+$ be nonzero elements so that $\bx_j = \bwy+\bz$. Then $\|\bwy\|, \|\bz\| < \|\bx_j\| = \lambda^+_j(L)$, and so the sets of vectors $\bx_1,\dots,\bx_{j-1},\bwy$ and $\bx_1,\dots,\bx_{j-1},\bz$ are linearly dependent. This means that
$$\bwy = \sum_{i=1}^{j-1} c_{1i} \bx_i,\ \bz = \sum_{i=1}^{j-1} c_{2i} \bx_i,$$
for some real coefficients $c_{11},\dots,c_{1(j-1)}$, not all zero, and $c_{21},\dots,c_{2(j-1)}$, also not all zero. Then
$$\bx_j = \bwy + \bz = \sum_{i=1}^{j-1} (c_{1i}+c_{2i}) \bx_i,$$
meaning that the set of vectors $\bx_1,\dots,\bx_j$ is linearly dependent. This is a contradiction, hence $\bx_j$ must be indecomposable.
\endproof

\noindent
As an immediate consequence of Lemma~\ref{s_min_ind} combined with~\eqref{s_min_1} we have a bound on the shortest indecomposable element in~$L^+$.

\begin{cor} \label{ind_bnd} $\min \left\{ \|\bx\| \in L^+ : \bx \text{ is indecomposable} \right\} \leq \sqrt{n} \left( 2\mu(L) + 1 \right)$.
\end{cor}
\medskip

\bigskip

\section{Indecomposables in planar lattices}
\label{ind_plane}

The goal of this section is to prove Theorem~\ref{ind_sum}. For this, we construct a cover of positive lattice points by indecomposable basis cones. Throughout this section, let $L \subset \real^2$ be a lattice of rank $2$ with positive successive minima $\lambda_1^+, \lambda_2^+$. Let $\bx,\bwy \in L^+$ be vectors corresponding to these positive successive minima.

\begin{lem} \label{sm_pb_2d} The vectors $\bx,\bwy$ form a basis for $L$.
\end{lem}

\proof
Let
$$P = \left\{ a \bx + b \bwy : 0 \leq a,b < 1 \right\}$$
be the parallelogram spanned by $\bx, \bwy$. As before, write $\BB_n(\lambda_2^+)$ for the closed ball of radius $\lambda_2^+$ centered at $\bo$ and let $\BB^o_n(\lambda_2^+)$ be its interior, the open ball. Notice that  $\BB^o_n(\lambda_2^+) \cap P$ cannot contain any point of $L$ linearly independent with $\bx$: this would contradict the definition of $\lambda_2^+$. Suppose $P$ contains a point of $L$, call it $\bz$, then $\bz \in P \setminus \BB^o_n(\lambda_2^+)$ as in Figure~\ref{PP}. Let $ABC$ be the triangle with the vertex $A$ at $\bx+\bwy$ and vertices $B,C$ at the intersection points of $\bar{P}$, the closure of $P$, and $\BB_n(\lambda_2^+)$, as in Figure~\ref{PP}. Then the point $D$ at $\bz$ is closer to the vertex $A$ than at least one of $B$ and $C$, since the furthest point from a vertex in a triangle is another vertex. The distance from $A$ to $D$ is $\|\bx+\bwy-\bz\|$, hence
$$\|\bx+\bwy-\bz\| < \max \{ AB, AC \} \leq \|\bwy\| = \lambda_2^+.$$
Since $\bx+\bwy-\bz \in L^+$ and it is linearly independent with $\bx$, this contradicts the definition of $\lambda_2^+$. This means that $P \cap L = \{ \bo \}$. Hence, by Lemma~\ref{empty}, $\bx$ and $\bwy$ form a positive basis for $L$.
\endproof

\begin{figure}
\centering
\includegraphics[scale=0.5]{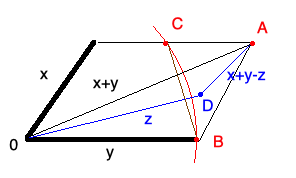}
\caption{Proof of Lemma~\ref{sm_pb_2d}: Parallelogram $P$ spanned by $\bx, \bwy$ with a vector $\bz$ in it. The red arc is the boundary of $\BB_n(\lambda_2^+)$ intersecting $P$.}\label{PP}
\end{figure}

\begin{cor} \label{sm_ind_2d} With notation as above, the cone $C(\bx,\bwy)$ contains no indecomposable elements of $L^+$ except for $\bx$ and $\bwy$. In particular, every point $\bz \in C(\bx,\bwy) \cap L^+$ can be expressed as $\bz = a \bx+ b \bwy$ for some $a,b \in \zed_{\geq 0}$.
\end{cor}

\proof
The statement follows immediately upon combining Lemmas~\ref{sm_pb_2d} and~\ref{cone_sem}.
\endproof
\smallskip

Let us write $\bx = (x_1,x_2)$ and $\bwy = (y_1,y_2)$, and assume without loss of generality that $x_1 > y_1$. Then we must have $x_2 < y_2$, since otherwise $\bx$ would be decomposable by Lemma~\ref{less}. Define the regions
\begin{equation}
\label{R1R2}
\R_1 = \left\{ \bz \in \real^2 : z_1 > x_1, z_2 < x_2 \right\},\ \R_2 = \left\{ \bz \in \real^2 : z_1 < y_1, z_2 > y_2 \right\}.
\end{equation}
Then all the indecomposables are contained in these two regions (see Figure~\ref{ind_1}). Let us say that indecomposables $\bu$ and $\bv$ are {\it consecutive} if the cone $C(\bu,\bv)$ contains no other indecomposables; for instance, $\bu$ and $\bv$ in $\R_2$ in Figure~\ref{ind_1} are consecutive. 

\begin{figure}
\centering
\includegraphics[scale=0.4]{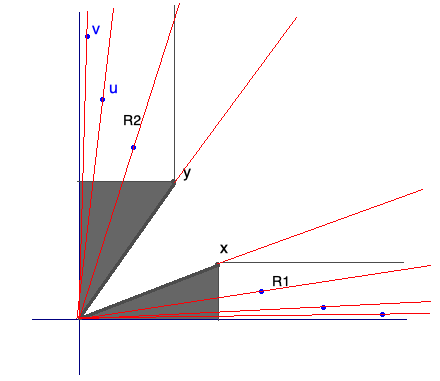}
\caption{Distribution of indecomposables: $\bx$, $\bwy$ and the blue dots. Red lines indicate consecutive indecomposable cones, the union of all of which covers $\real^2_{\geq 0}$.}\label{ind_1}
\end{figure}

\begin{lem} \label{ind_basis} Let $\bu$ and $\bv$ be consecutive indecomposables. Then they form a basis for $L$.
\end{lem}

\proof
Our proof is illustrated by Figure~\ref{ind_2}. If the consecutive indecomposables are $\bx,\bwy$, then the result follows from Lemma~\ref{sm_pb_2d}. Otherwise, without loss of generality, assume that $\bu,\bv \in \R_2 \cup \{\bwy\}$, $u_2 < v_2$, and let $P$ be the parallelogram spanned by $\bu$ and $\bv$ (the argument is completely analogous if $\bu,\bv, \in \R_1 \cup \{\bx\}$). We will prove that $P$ contains no point of the lattice $L$ in its interior. Suppose it does, call this point $\bz$. Then $\bz$ is decomposable, since otherwise $\bu$ and $\bv$ would not be consecutive. Define the region
$$R = \left\{ \bw \in P : w_2 \leq v_2 \right\},$$
and suppose that $\bz \in R$. Notice that we must have $z_2 > u_2$, since otherwise we would have $\bz < \bu$ contradicting the fact that $\bu$ is indecomposable (Lemma~\ref{less}). Additionally, we must have $z_1 < u_1$. Indeed, suppose $z_1 \geq u_1$, then $\bz > \bu$, so let $\bq = \bz - \bu$. This $\bq = (q_1,q_2) \in L^+$, where 
$$q_2 = z_2 - u_2 \leq v_2-u_2 < v_2.$$
Then $q_1 = z_1-u_1 > v_1$, since otherwise we would have $\bq < \bv$, which contradicts $\bv$ being indecomposable (Lemma~\ref{less}). Then $z_1 > u_1+v_1$, but $\bz \in P$ and every point in $P$ is of the form $a\bu+b\bv$ with $a,b < 1$, meaning that $z_1 < u_1+v_1$, a contradiction. Then we have $\bz \in P$ with
$$z_1 < u_1,\ u_2 < z_2 \leq v_2.$$
Since $\bz$ is decomposable, it must be a positive integer linear combination of some finite collection of indecomposables (Lemma~\ref{no_fin}). In particular, there must exist some indecomposable element $\bt$ such that $\bt < \bz$. If $t_2 < u_2$, then $\bt < \bu$, again contradicting the fact that $\bu$ is indecomposable (Lemma~\ref{less}); by the same reasoning, we must have $t_1 > v_1$, since otherwise $\bt < \bv$, contradicting the fact that $\bv$ is indecomposable. But then we must have 
$$v_1 < t_1 < z_1 < u_1,\ u_2 < t_2 < z_2 < v_2,$$
and so $\bt \in P$, which contradicts the fact that $\bu$ and $\bv$ are consecutive indecomposables. Hence, $\bz \notin R$, so suppose $\bz \in P \setminus R$. Then we have
$$\bv < \bz < \bv+\bu,$$
and so $\bz-\bv, \bv+\bu-\bz \in L^+$. On the other hand, $\bu = (\bz-\bv) + (\bv+\bu-\bz)$, which contradicts the fact that $\bu$ is indecomposable. Thus, we proved that $P$ contains no point of the lattice $L$ in its interior, and so $\bu,\bv$ is a basis for $L$ by Lemma~\ref{empty}.
\endproof

\begin{figure}
\centering
\includegraphics[scale=0.4]{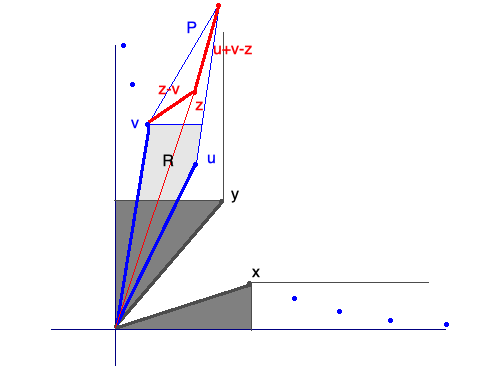}
\caption{Proof of Lemma~\ref{ind_basis}.}\label{ind_2}
\end{figure}

Next, we want to prove that $\real^2_{\geq 0}$ is covered by the union of indecomposable cones. Without loss of generality, let us focus on consecutive indecomposables in the region $\R_2$, as defined in~\eqref{R1R2} and illustrated in Figure~\ref{ind_1} (the situation with those in $\R_1$ is completely analogous). Let us order indecomposables in $\R_2$ as $\bwy = \bu_1, \bu_2, \bu_3, \dots$ in the order of increasing sup-norm, i.e., $|\bu_i| < |\bu_{i+1}|$ and $\bu_i,\bu_{i+1}$ is a consecutive pair of indecomposables for every $i \geq 1$.

\begin{lem} \label{cone_cover} Define $\C(\R_2) = \{ a \bt : a \geq 0, \bt \in \R_2 \}$. Then
$$\C(\R_2) = \bigcup_{i=1}^{\infty} C(\bu_i,\bu_{i+1}).$$
\end{lem}

\proof
Suppose that $\R_2$ contains some indecomposable $\bu_m$ which is a scalar multiple of the standard basis vector~$\be_2$. Then $\R_2$ contains only finitely many indecomposables $\bu_1,\dots,\bu_m$, and so $\R_2 = \bigcup_{i=1}^{m-1} C(\bu_i,\bu_{i+1})$. 

Assume this is not the case, then $u_{i1} > 0$ for any $i$. First notice that $u_{i1} > u_{i+1}$ for every $i \geq 1$, by Lemma~\ref{less}. Let us prove that $\lim_{i \to \infty} u_{i1} = 0$; in particular,  $\R_2$ contains infinitely many indecomposables. Suppose not, then there exists some $\eps > 0$ such that $u_{i1} > \eps$ for all $i \geq 1$. This means that the vertical strip 
$$S = \{ \ba \in \real^2_{\geq 0} : a_1 \leq \eps \}$$
does not contain any indecomposable element, and thus any nonzero point of the lattice $L$. Hence, if $\alpha x_1 + \beta y_1 \geq 0$ for integers $\alpha,\beta$, not both zero, then
$$\alpha x_1 + \beta y_1 > \eps.$$
In particular, $\L_{x_1,y_1}(\alpha,\beta) := \alpha x_1 + \beta y_1 \neq 0$, so $x_1,y_1$ must be $\que$-linearly independent, which means that $x_1/y_1$ is an irrational real number. Then Kronecker's approximation theorem (see, e.g., \cite{cass:dioph}) states that the sequence of fractional parts $\{ \alpha (x_1/y_1) \}$ as $\alpha$ ranges through $\zed$ is dense in the interval $[0,1)$. Specifically, this means that for any $\eps > 0$ there exist $\alpha,\beta \in \zed$ such that $0 < \beta + \alpha (x_1/y_1) < \eps/y_1$, hence
$$0 < \L_{x_1,y_1}(\alpha,\beta) < \eps.$$
This contradicts the assumption that $S \cap L = \{ \bo \}$, and so $\lim_{i \to \infty} u_{i1} = 0$; hence, there must be infinitely many indecomposables of $L$ in $\R_2$. Additionally, $\lim_{i \to \infty} u_{i2} = \infty$. Suppose not, then there exists some real number $R$ so that $u_{i2} \leq R$ for all $\bu_i \in \R_2$. This implies that $|\bu_i| \leq R$ since $u_{i1}$ tends to $0$, and hence there can be only finitely many such $\bu_i$ since $L$ is discrete.

Now, let $\bz \in \C(\R_2)$. There must exist some two consecutive indecomposables $\bu_i, \bu_{i+1}$ such that 
$$\frac{u_{i2}}{u_{i1}} \leq \frac{z_2}{z_1} \leq \frac{u_{(i+1)2}}{u_{(i+1)1}},$$
i.e., the slope of the line through the origin and $\bz$ is between the slopes of the lines through the origin and $\bu_i,\bu_{i+1}$, respectively. Therefore, $\bz \in C(\bu_i,\bu_{i+1})$, and thus $\C(\R_2)$ is covered by the union of consecutive indecomposable cones.
\endproof

\begin{rem} \label{cover} The same argument as in the proof of Lemma~\ref{cone_cover} applies to the region $\R_1$. Hence, $\real^2_{\geq 0} = \C(\R_1) \cup C(\bx,\bwy) \cup \C(R_2)$ is covered by the union of all the consecutive indecomposable cones.
\end{rem}

We are now ready for the main result of this section.

\proof[Proof of Theorem~\ref{ind_sum}]
Since $\real^2_{\geq 0}$ is covered by the union of consecutive indecomposable cones (Remark~\ref{cover}), there must exist some consecutive indecomposables $\bu',\bv'$ such that $\bz \in C(\bu',\bv')$. First, assume that $\bz$ lies on the boundary of this cone, then it is a multiple of either $\bu'$ or $\bv'$, say, $\bz = a \bu'$. Then $a \geq 2$ is an integer since $\bz$ is decomposable. In this case, choose $\bu = \bv = \bu'$ and take $\alpha = a-1 > 0$, $\beta = 1$, so
$$\bz = \alpha \bu + \beta \bv.$$
Next, suppose that $\bz$ is in the interior of this cone. Then take $\bu = \bu'$, $\bv = \bv'$. Since $\bu,\bv$ form a basis for $L$ (Lemma~\ref{ind_basis}), there must exist positive integers $\alpha,\beta$ such that $\bz = \alpha \bu + \beta \bv$, by Lemma~\ref{cone_sem}. In either of these two cases, $\bz, \bu, \bv > \bo$ and $\alpha,\beta > 0$, hence we have $|\bu| + |\bv| = |\bu+\bv| \leq |\bz|$.
\endproof

\begin{rem} \label{hk} After producing our proof of the above theorem, we became aware of the nice work of Hejda and Kala~\cite{hejda_kala}. Theorem~2 of that paper provides a similar result to our Theorem~\ref{ind_sum} for indecomposables in the ring of integers of a real quadratic number field, whereas our theorem applies to any planar lattice. Further, a pair of consecutive indecomposables in a lattice $L$ forms a {\it minimal system} in the sense of Voronoi (see~\cite{ustinov} for details); hence, our Theorem~\ref{ind_sum} provides a slightly different perspective on that classical subject.
\end{rem}

We finish this section by a couple brief remarks on the differences between pairs of consecutive indecomposables. Given a consecutive pair of indecomposables $\bu,\bv$, we define their difference vector $\bd(\bu,\bv) = \bv-\bu$. We first prove that all such difference vectors are distinct.

\begin{lem} \label{dist_diff} Let $\bu,\bv$ be a pair of consecutive indecomposables in the planar lattice $L$ and $\bd(\bu,\bv)$ the corresponding difference vector. Then there exist at most finitely many other pairs of consecutive indecomposables $\bt,\bw$ in $L$ such that $\bd(\bu,\bv) = \alpha \bd(\bt,\bw)$ for some $\alpha \in \real$.
\end{lem}

\proof
Recall that the {\it Klein polygon} of $L$ is the convex hull of $L^+$. The indecomposable elements in a planar lattice $L$ are either vertices or points on the edges of the Klein polygon of $L$ (see~\cite{ustinov}; this is not necessarily the case in higher dimensions, as indicated in Example~2.8 of~\cite{kala_man}). Difference vectors of consecutive pairs of indecomposables are scalar multiples of each other if and only if these pairs lie on the same edge, but each edge can contain only finitely many indecomposables (see \cite{hejda_kala} for examples of finite sequences of such pairs of consecutive indecomposables).
\endproof

We also briefly consider the distribution of these difference vectors. 

\begin{lem} \label{diff_size} Assume the lattice $L$ is not PVR. Let us write $|\bd(\bu_i,\bu_{i+1})|$ and $\min \bd(\bu_i,\bu_{i+1})$ for the maximum and minimum absolute value of coordinates of the difference vector $\bd(\bu_i,\bu_{i+1})$, respectively. Then
$$\limsup_{i \to \infty} |\bd(\bu_i,\bu_{i+1})| = \infty,\ \liminf_{i \to \infty} \min \bd(\bu_i,\bu_{i+1}) = 0.$$
\end{lem}

\proof
Since $L$ is not PVR, at least one of the regions $\R_1$ and $\R_2$ must contain infinitely many indecomposables; without loss of generality, assume it is $\R_2$. Let us write $d_{i1},d_{i2}$ for the coordinates of the difference vector $\bd(\bu_i,\bu_{i+1})$, then $|d_{i1}| = \min \bd(\bu_i,\bu_{i+1})$ and $d_{i2} = |\bd(\bu_i,\bu_{i+1})|$, since $\bu_i,\bu_{i+1} \in \R_2$. In particular, $y_1 > |d_{i1}|$ for every $i \geq 1$. Suppose that $\liminf_{i \to \infty} \min \bd(\bu_i,\bu_{i+1}) \neq 0$, then there exists some $\eps > 0$ such that $|d_{i1}| > \eps$ for all $i \geq 1$. This means that $u_{i1} < y_1 - \eps (i-1) < 0$ for sufficiently large $i$. This contradicts $\bu_i$ being in $L^+$.

On the other hand, suppose that $\limsup_{i \to \infty} |\bd(\bu_i,\bu_{i+1})| \neq \infty$, i.e., there exists some real number $T > 0$ such that $d_{i2} \leq T$. Since each difference vector can be equal to at most finitely many others, by Lemma~\ref{dist_diff}, there must be infinitely vectors of $L$ satisfying $|\bd(\bu_i,\bu_{i+1})| \leq T$. This contradicts discreteness of~$L$.
\endproof

\bigskip

\section{Counting indecomposables from real quadratic number fields}
\label{counting}

The convex hull of indecomposables in $L^+$ for a planar lattice $L$ forms a Klein polygon approximating the coordinate axes as close as possible by positive points of the given lattice. This suggests that the number of indecomposables of bounded sup-norm in $L$ grows logarithmically as this bound tends to infinity. We aim to produce a precise version of this observation for lattices coming from the rings of integers of real quadratic fields.

The main goal of this section is to prove Theorem~\ref{main_cnt}. Let $K$ be a real quadratic field with ring of integers $\O_K$ and discriminant $\Delta_K$. Let $J \subseteq \O_K$ be an ideal. Let $L = \Sigma_K(J)$ and observe that for any $\bo \neq \bz = (z_1,z_2) \in L$, $z_1,z_2 \neq 0$. Hence, $L$ contains infinitely many indecomposables in both regions $\R_1$ and $\R_2$, defined as in~\eqref{R1R2}. The norm of $J$ is then given by $\Nn_K(J) = |\O_K:J|$ and the determinant of the lattice $L$ is
\begin{equation}
\label{norm_det}
\det(L) = \Nn_K(J) \sqrt{|\Delta_K|}.
\end{equation}
The following lemma is an adaptation of Theorem~5 of~\cite{kala_yatsyna} along with its proof.

\begin{lem} \label{log_curve} Let $L = \Sigma_K(J)$ be as above. Let $\bu = (u_1,u_2) \in L^+$ be an indecomposable element. Then
$$u_1 u_2 \leq \det(L)^2.$$
\end{lem}

\proof
Suppose that $u_1 u_2 > \det(L)^2$, then there exists $\eps > 0$ small enough so that
$$(\sqrt{u_1} - \eps)(\sqrt{u_2} - \eps) > \det(L).$$
Define a box
$$B_{\eps} = \left\{ \bx \in \real^2 : |x_i| \leq \sqrt{u_i} - \eps, i = 1,2 \right\},$$
then $\Vol_2(B_{\eps}) = 2^2 (\sqrt{u_1} - \eps)(\sqrt{u_2} - \eps) > 2^2 \det(L)$. Hence, Minkowski Convex Body Theorem (see, e.g.,~\cite{grub:lek}) guarantees that there exists a nonzero point $\bwy \in B_{\eps} \cap L$. This means that $\bwy = (\sigma_1(\beta),\sigma_2(\beta))$ for some $0 \neq \beta \in J$, so $\beta^2 \in J$. Then 
$$\bz = (z_1,z_2) = (\sigma_1(\beta)^2, \sigma_2(\beta)^2) = \Sigma_K(\beta^2) \in L^+,$$ 
and $z_1 < u_1$, $z_2 < u_2$. Then Lemma~\ref{less} implies that $\bu$ is not indecomposable. This completes the proof.
\endproof

Define the region
$$R_L = \left\{ \bx \in \real^2_{\geq 0} : x_1x_2 \leq \det(L)^2 \right\},$$
then, by Lemma~\ref{log_curve}, all the indecomposable elements of $L$ lie in $R_L$. For any $T \geq 0$, define $R_L(T) = \{ \bx \in R_L : |\bx| \leq T \}$. Then 
\begin{eqnarray}
\label{arlt}
\Ar(R_L(T)) & = & \det(L)^2 + \int_{\frac{\det(L)^2}{T}}^T \frac{\det(L)^2}{x} dx \nonumber \\
& = & \det(L)^2 (2 \log (\det(L)^2 T) + 1).
\end{eqnarray}
Suppose that $R_L(T)$ contains $n$ indecomposables $\bu_1,\dots,\bu_n$, written in the order of decreasing first coordinate. Then for each $1 \leq i \leq n-1$, $\bu_i,\bu_{i+1}$ is a consecutive pair of indecomposables, hence, a basis for $L$. Let $\bu_i', \bu_{i+1}'$ be the intersection points of the lines $t\bu_i, t\bu_{i+1}$, respectively, with the curve $y= \frac{\det(L)^2}{x}$. Let $S_i, S_i'$ be the triangles, which are the convex hulls of $\bo,\bu_i,\bu_{i+1}$ and $\bu_i',\bu_{i+1}'$, respectively, so $S_i \subseteq S_i'$. Then
\begin{equation}
\label{si_ar}
\Ar(S_i') \geq \Ar(S_i) = \frac{\det(L)}{2},
\end{equation}
since consecutive indecomposables $\bu_i,\bu_{i+1}$ form a basis for $L$ (Lemma~\ref{ind_basis}), and hence, the parallelogram spanned by $\bu_i,\bu_{i+1}$ is a fundamental domain for $L$. Write also $C_i$ for the intersection of the cone $C(\bu_i',\bu_{i+1}')$ with $R_L(T)$, then $C_i \subset S_i'$. In our next lemma, we estimate the area of the region $S_i \setminus C_i$, an example of which is colored yellow in Figure~\ref{log_ind_fig}.

\begin{lem} \label{area_error} Let us write $D$ for $\det(L)$. Then
$$\Ar(S_i \setminus C_i) \leq \frac{D}{2} - D^2 \log \left( \frac{1+\sqrt{4D^2+1}}{2D} \right).$$
\end{lem}

\proof
First notice that $A_i := \Ar(S_i \setminus C_i)$ is maximized when $S_i = S_i'$, i.e., when the indecomposables $\bu_i,\bu_{i+1}$ lie on the hyperbola $x_2=D^2/x_1$. From here on, assume this is the case. Let $\bu_i = (u_{i1},u_{i2})$, $\bu_{i+1} = (u_{(i+1)1},u_{(i+1)2})$, then
\begin{equation}
\label{ae_1}
u_{i1} u_{i2} = u_{(i+1)1} u_{(i+1)2} = D^2,
\end{equation}
and
\begin{equation}
\label{ae_2}
D = \left| \det \begin{pmatrix} u_{i1} & u_{(i+1)1} \\ u_{i2} & u_{(i+1)2} \end{pmatrix} \right| = \left| u_{i1} u_{(i+1)2} - u_{(i+1)1} u_{i2} \right|.
\end{equation}
Notice that $A_i$ is the area of the region between the hyperbola $x_2=D^2/x_1$ and the cord connecting the points $\bu_i, \bu_{i+1}$. Setting up the integral and using condition~\eqref{ae_1}, we find that
\begin{equation}
\label{Ai}
A_i = D^2 \left( \frac{u_{i1}^2 - u_{(i+1)1}^2}{2u_{i1} u_{(i+1)1}} - \log \left( \frac{u_{i1}}{u_{(i+1)1}} \right) \right).
\end{equation}
Let us write $R = u_{i1}/u_{(i+1)1}$, then $R > 1$. Substituting $u_{i2} = D^2/u_{i1}$ and $u_{(i+1)2} = D^2/u_{(i+1)1}$ from~\eqref{ae_1} into~\eqref{ae_2}, we obtain
$$R - 1/R = 1/D,$$
since $R > 1/R$. Solving this equation for $R$, we obtain $R = \frac{1+\sqrt{4D^2+1}}{2D}$. Now, substitute it in for $u_{i1}/u_{(i+1)1}$ into~\eqref{Ai} to obtain
$$A_i = D^2 \left( \frac{R}{2} - \frac{1}{2R} - \log R \right) = D^2 \left( \frac{1}{2D} - \log \left( \frac{1+\sqrt{4D^2+1}}{2D} \right) \right).$$
\endproof

\proof[Proof of Theorem~\ref{main_cnt}]

Notice that
\begin{equation}
\label{area_sum}
\sum_{i=1}^{n-1} \Ar(C_i) \leq \Ar(R_L(T)),
\end{equation}
where
$$\Ar(C_i) = \Ar(S_i) - \Ar(S_i \setminus C_i) \geq D^2 \log \left( \frac{1+\sqrt{4D^2+1}}{2D} \right),$$
by combining~\eqref{si_ar} with Lemma~\ref{area_error}. Now, combining this last observation with~\eqref{area_sum} and~\eqref{arlt}, we obtain
$$n \leq \left( \log \left( \frac{1+\sqrt{4D^2+1}}{2D} \right) \right)^{-1} (2 \log (D^2 T) + 1) + 1.$$
\endproof

\begin{figure}
\centering
\includegraphics[scale=0.4]{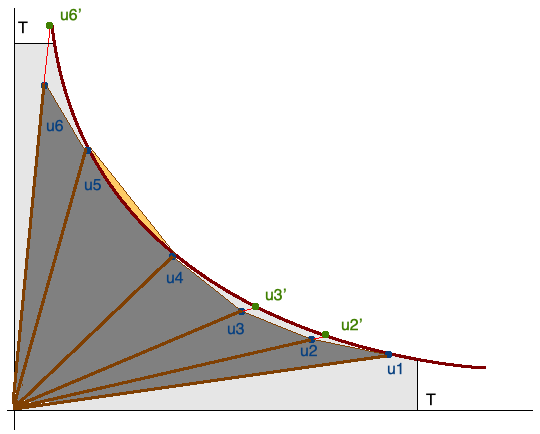}
\caption{Region $R_L(T)$ with indecomposables in it.}\label{log_ind_fig}
\end{figure}

\bigskip

\noindent
{\bf Acknowledgement:} We thank Vitezslav Kala, Ethan Schlegel and Mikulas Zindulka for some helpful comments which improved our paper. We are also grateful to the referee for useful remarks and corrections.
\bigskip

\bibliographystyle{plain}  

\end{document}